\documentclass[10pt]{amsart}
\usepackage[bookmarksopen,bookmarksdepth=2]{hyperref}
\usepackage{amssymb,amsmath,amsfonts,amsthm,epsfig,amscd,stmaryrd}
\usepackage{a4wide}

\def\Zhat{\widehat{\Z}}
\def\P{\mathfrak{p}}
\def\Vbar{\overline{V}}
\def\Kbar{\overline{K}}
\def\QQ{E}
\def\GL{\mathrm{GL}}
\def\p{\mathfrak{p}}
\def\PP{\mathbf{P}}
\def\PSL{\mathrm{PSL}}
\def\R{\mathbf{R}} 
\def\Z{\mathbf{Z}}
\def\Q{\mathbf{Q}}

\def\F{\mathbf{F}}
\def\Z{\mathbf{Z}}

\def\R{\mathbf{R}}

\def\OL{\mathcal{O}}
\def\Frob{\mathrm{Frob}}
\def\C{\mathbf{C}}
\def\Gal{\mathrm{Gal}}
\def\SL{\mathrm{SL}}
\def\H{\mathbf{H}}

\usepackage{xcolor}

\renewcommand\mathbb{\mathbf}

\newtheorem{theorem}[subsection]{Theorem}
\newtheorem{lemma}[subsection]{Lemma}
\newtheorem{df}[subsection]{Definition}

\theoremstyle{remark}
\newtheorem{remark}[subsection]{Remark}

\begin{document}

\title{The adelic closure of  triangle groups}

\author[F.~Calegari]{Frank Calegari}  \email{fcale@math.uchicago.edu} \address{The University of Chicago,
5734 S University Ave,
Chicago, IL 60637, USA}

 \thanks{The author was supported in part by NSF Grants DMS-2001097 and DMS-2450123.}

\begin{abstract}
Motivated by questions arising from billiard trajectories in the regular $n$-gon, McMullen~\cite{CM} defined a pair of functions~$\kappa$ and $\delta$ on the cusps~$c$ of the corresponding triangle group~$\Delta_n$ inside $\mathrm{SL}_2({\mathcal{O}})$, where ${\mathcal{O}} = \Z[\zeta_n+ \zeta^{-1}_n]$. McMullen asks for which $n$ these functions are \emph{congruence}, that is,  when they only depend on the image of the cusp $c \in \mathbf{P}^1(\mathcal{O})$ in $\mathbf{P}^1(\mathcal{O}/N)$ for some integer~$N$. In this note, we  answer McMullen's questions. We obtain our results by computing the exact closure of $\Delta_n \subset \mathrm{SL}_2({\mathcal{O}})$ inside $\SL_2(\widehat{{\mathcal{O}}})$, where $\widehat{{\mathcal{O}}}$ is the profinite completion of ${\mathcal{O}}$.
\end{abstract}

\maketitle

\def\wDelta{\widehat{\Delta}}
\def\wGamma{\widehat{\Gamma}}
\def\wQ{\widetilde{Q}}

\section{Introduction}

The goal of this note is to answer some questions concerning triangle groups first
raised in~\cite{CM}.
Fix an integer~$n \ge 3$, and let~$\zeta = e^{2 \pi i/n}$
be an~$n$th root of unity in~$\C$.
In this paper, we shall be concerned with the family of triangle groups~$\Delta = \Delta_n \subset \SL_2(\R)$
given explicitly by~$\Delta = \langle T,U,R \rangle$, where:
$$T = \left( \begin{matrix} 1 & 2 + \zeta + \zeta^{-1} \\
0 & 1 \end{matrix} \right), \quad
U = - \left( \begin{matrix} 1 & 0 \\ -1 &  1 \end{matrix} \right), \quad
R = UT.$$
The element~$R$ has order~$n$, and the quotient~$\H/\Delta$ is
a hyperbolic orbifold of genus zero with signature~$(n/2,\infty,\infty)$ when~$n$
is even and~$(2,n,\infty)$ when~$n$ is odd.
If~$K = \mathbf{Q}(\zeta+\zeta^{-1})$ and~$\OL = \OL_K=\Z[\zeta+\zeta^{-1}]$,  then~$\Delta \subset \Gamma:=\SL_2(\OL)$.
Let~$\widehat{\OL}$ denote the closure of~$\OL$ inside the finite adeles~$\mathbf{A}^{\kern-0.1em{f}}_K$.
 There is a natural map
$$\Delta \rightarrow \SL_2(\OL) \hookrightarrow \SL_2(\widehat{\OL}).$$
Let~$\wDelta$  denote the closure of the image of~$\Delta$.
The adelic group~$\wGamma:=\SL_2(\widehat{\OL})$ has a natural topology as a profinite group; we can
also view it as  the congruence closure of~$\SL_2(\OL)$, that is, the profinite group given by the projective
limit of the groups~$\SL_2(\OL/d)$ over all integers~$d$, or equivalently, the
limit of the groups~$\SL_2(\OL/\alpha)$ over $\alpha \in \OL$. These groups are cofinal
with respect to each other since there is a surjection $\SL_2(\OL/d) \rightarrow \SL_2(\OL/\alpha)$ with $d = N_{K/\Q}(\alpha)$.
As noted in~\cite[Prop~1.15]{CM}, the group~$\wDelta$ has finite index in~$\wGamma$.
The goal of this note is to describe~$\wDelta$ 
as precisely as possible.

\begin{theorem} \label{awayfrom2} 
The congruence closure~$\wDelta$  of~$\Delta$ is a direct product of the  closures
$$\wDelta_p
\subseteq \wGamma_p = \SL_2(\OL \otimes \Z_p) = \prod_{\p | p} \SL_2(\OL_{\p})$$
 of~$\Delta$ for all primes~$p$ of~$\Z$.
Moreover, ~$\wDelta_p$ has index one for all odd primes~$p$ outside the following  two exceptions,
where in both cases there is a unique prime~$\p$ above~$p$ in~$\OL$.
\begin{enumerate}
\item $n = 2p^k$, $k \ge 1$,  and~$\wDelta_p$ is the preimage 
of~$\Z/p\Z \subset \SL_2(\OL/\p) = \SL_2(\F_p)$ 
with index~$p^2 - 1$. 
\item $n = 5p^k$, $k \ge 0$, $p = 3$, and~$\wDelta_p$ is the preimage of~$A_5 \subset 
A_6 \simeq \PSL_2(\F_9) = \PSL_2(\OL/\p)$.
 \end{enumerate}
\end{theorem}

This follows from Lemmas~\ref{pbyp},~\ref{mnot2}, and~\ref{ramp}.
Theorem~\ref{awayfrom2} determines~$\wDelta$
 up to a precise description of the~$2$-adic closure~$\wDelta_2$,
which we turn to now.
In this introduction, we only give the statement in the case when~$n$ is not a  power of~$2$,
and leave the statement of the remaining case when~$n = 2^k$
to the main body of the text (see Theorem~\ref{2powertheorem}).
We first recall some standard definitions:

\begin{df} \label{defJ}
For any (commutative) ring~$A$, any ideal~$I$, and any group~$G \subset \SL_2(A)$,  let~$G(J)$ denote the kernel of the homomorphism~$G \rightarrow \SL_2(A/J)$.
Furthermore, we let~$M^0(A)$ denote the additive group of~$2 \times 2$ matrices over~$A$ with trace zero.
\end{df}

If~$A$ is a ring with~$2=0$, then~$M^0(A)$ has a subgroup isomorphic to~$A$ consisting of diagonal matrices.
The following is proved in Section~\ref{2adicfirst} (note that Theorem~\ref{2part}~(\ref{cee}) is~\cite[Thm~1.1]{CM}).

\begin{theorem} \label{2part}
Suppose that~$n$ is not a power of~$2$. Then:
\begin{enumerate}
\item There is an equality~$\wDelta(4)_2 = \wGamma(4)_2 = \SL_2(\OL \otimes \Z_2)(4)$.
\item Consider the homomorphism
$$\wDelta(2)/\wDelta(4) \simeq \Delta(2)/\Delta(4) \rightarrow 
\Gamma(2)/\Gamma(4) = 
I + 2 M^0(\OL/2).$$
The index of the image is~$|\OL/2|/2$. The intersection of the image with the
 subgroup~$\OL/2$ of diagonal matrices of the target consists only of~$\{0,I\}$.
 \item  \label{cee} The image of~$\Delta/\Delta(2)$ inside~$\SL_2(\OL/2)$ is isomorphic to~$D_{n'}$
 where~$n' = n/2$ if~$n$ is even and~$n' = n$ otherwise.
\end{enumerate}
In particular, the~$2$-adic index is given by
$\displaystyle{ \frac{|\SL_2(\OL/2)|}{2n'} \cdot \frac{|\OL/2|}{2}}$.
\end{theorem}

 One application
of these results is to answer a number of questions first raised in~\cite{CM}.
(The paper~\cite{CM} is itself an offshoot of a number of other papers --- including~\cite{BilliardsRegular} and~\cite{BilliardsHeights} --- whose goal is to understand billiards
in regular polygons.)  One of the intentions of~\cite{CM} was to distill a number of purely group theoretical
and number theoretical questions in a transparent way, and~\cite{CM} makes no direct reference
to billiards. In this paper,
we similarly shall not mention these connections, but rather refer the reader 
(as a starting point) to~\cite{CM,BilliardsRegular,BilliardsHeights}.
Instead, the questions raised in~\cite{CM} are phrased directly in terms of the arithmetic of the cusps
of~$\Delta$, and we now recall the basic story of these cusps.
The group~$\Delta \subset \SL_2(\R)$ acts on the usual upper
half plane~$\H$ with finite covolume. The quotient curve (orbifold) $\H/\Delta$ has either one or two cusps
($\infty$ and~$0$) depending on the parity of~$n$; in the usual way
one  can think about these cusps as certain orbits of~$\PP^1(K) \subset \PP^1(\R)$ under the action of~$\Delta$.
Very explicitly, if
\begin{equation}
\label{gamma}
\gamma = \left( \begin{matrix} a & b \\ c & d \end{matrix} \right) \in \Delta,
\end{equation}
then~$a/c \in \PP^1(K)$ lies in the orbit of~$\infty$ and~$b/d$ lies in the orbit of~$0$.
If~$[K:\Q] \le 2$, then every point of~$\PP^1(K)$ is a cusp. If~$[K:\Q] > 2$,
this is no longer true, and there is no known explicit characterization of
the cusps. Even when~$[K:\Q] = 2$
(which occurs for~$n=5$, $8$, and~$12$),
there are still many mysteries regarding the cusps.  In these cases, given~$x \in \PP^1(K)$, then there exists a~$\gamma \in \Delta$
as in~(\ref{gamma}) such that either~$a/c$ or~$b/d$ is equal to~$x$. (For more general~$n$ this still holds at least if
one assumes that~$x$ is a cusp of~$\Delta$.) But the column~$[a,c]$ of~$\gamma$ turns out to be unique up to sign,
whereas (assuming~$[K:\Q] > 1$ so the unit group of~$K$ is infinite) there are many possible pairs~$(a,c)$ for which~$a/c=x$. It is precisely the map which sends a cusp~$x \in \PP^1(K)$ to a representative~$[a,c]$ (up to sign)
with~$a/c = x$ which is poorly understood.
In~\cite{CM}, certain simple functions~$\delta$ and~$\kappa$ on the set of cusps
are considered.
For example, when~$n$ is even, the map~$\kappa$ is the function on the set of cusps~$x \in \PP^1(K)$ which is~$0$
when~$x$ is in the orbit of the cusp~$\infty$ and~$1$ if~$x$ is in the orbit of~$0$. The definition of~$\delta$
is slightly more complicated and is recalled in section~\ref{deltakappa}.
McMullen raises
 the problem
 as to what extent these maps are \emph{congruence}, that is, whether they only depend on the image
of~$x \in \PP^1(\OL)$ in~$\PP^1(\OL/N)$ for some integer~$N$.
It is shown in~\cite{CM} that, when the prime factorization of~$n$ has certain explicit forms, these maps are congruence.
However, it is also shown~\cite[Thm~1.13]{CM}
that~$\delta$ is \emph{not} congruence when~$n = 12$.
In this paper, we extend this result to show that~$\delta$ and~$\kappa$ are \emph{never} congruence invariants in all
the remaining cases not addressed in~\cite{CM}, generalizing the result (and method) for~$\delta$ addressed
in~\cite{CM} when~$n=12$. In particular, we have (see Theorems~\ref{maindeltatwo} and~\ref{kappacong}
and \S ~\ref{curtadded}):

\begin{theorem} \label{curty} The invariant~$\kappa$ is congruence unless~$n=2m$ where~$m$ is odd and has at least two
prime factors.
The invariant~$\delta$ 
as a function on the cusps equivalent to $\infty$
is congruence if and only if either~$4 \nmid n$
or~$n = 2^k$.
The invariant~$\delta$  as a function
on all the cusps is congruence if and only if  either $n = 2^k$,
$n = 2 p^k$ for an odd prime $p$, or $n$ is odd.
\end{theorem}

As suggested by the arguments of~\cite{CM}, 
Theorem~\ref{curty} can be deduced from
a sufficiently precise understanding
of the adelic closure~$\wDelta$ of~$\Delta$ in~$\wGamma = \SL_2(\widehat{\OL})$
provided by Theorem~\ref{awayfrom2} and~\ref{2part}.

\begin{remark}
If $E$ is an elliptic curve over a number field $F$, then, by considering
the collection of $p$-adic Galois representations associated to the Tate
module $T_p(E)$ for each prime $p$, one obtains a representation:
$$\rho_E: G_{\Q} \rightarrow \GL_2(\Zhat).$$
In~\cite{Serre}, Serre proved that the image of~$\rho_E$ has \emph{finite index} in 
$\GL_2(\Zhat)$. Similar Galois representations
arise  from
the theory of modular forms (and its generalizations).
 The problem of precisely determining the images of such representations
 goes back at least 50 years --- see in particular~\cite{SWD}.
 The main theorem of this paper shares many similarities with big image
theorems in the theory of Galois representations, and so unsurprisingly a 
number of our methods are quite similar to those of~\cite{SWD} and its generalizations.
The analogy between these problems was already noted in~\cite{CM}.
\end{remark}

Combining Theorems~\ref{awayfrom2} and~\ref{2part} (together with Theorem~\ref{2powertheorem} which explains what
happens when~$n=2^k$), we obtain the following formula for the full index of~$\wDelta$ in~$\wGamma$, namely:

\begin{theorem} The index of~$\wDelta$ in~$\wGamma$ for~$n \ge 3$ is given by
$\displaystyle{ \frac{|\SL_2(\OL/2)|}{2n'} \cdot \frac{|\OL/2|}{2}}$ 
where~$n'=n/2$ if~$n$ is even and~$n$ if~$n$ is odd,
times the following extra factors in the  following exceptional circumstances:
\begin{enumerate}
\item A factor of~$2$ if~$n = 2^k$.
\item A factor of~$6$ if~$n =5 \cdot 3^k$ and~$k \ge 0$.
\item A factor of~$p^2-1$ if~$n = 2 \cdot p^k$ and~$k \ge 1$.
\end{enumerate}
\end{theorem}

Note that at most one of these exceptional circumstances can occur for any given~$n$.

\begin{remark}
We can also write out the order of~$|\SL_2(\OL/2)| \cdot |\OL/2|$
more explicitly as follows:
Let~$d$ be the degree of~$K$, so~$d = \varphi(n)/2$.
With~$n = 2^k \cdot m$ with~$m$ odd, and~$n \ge 3$.
\begin{enumerate}
\item Let $e$ be the ramification degree of~$2$ in~$K$, so, if~$m > 1$, then~$e = \varphi(2^k)$,
and if~$m=1$, then~$e = \varphi(n)/2 = \varphi(2^{k})/2$.
\item Let $f$ be the order of the decomposition group of~$2$ in~$K$, so~$f$ is the order of~$2$ in~$(\Z/m\Z)^{\times}/(\pm 1)$.
\item Let $r$ be the number of primes above~$2$ in~$\OL$, so~$d = er\kern-0.1em{f}$. 
\end{enumerate}
Then~$|\OL/2| = 2^d$, and
$$
\begin{aligned}
|\SL_2(\OL/2)| = & \  |\F|^{3r(e-1)} \cdot |\SL_2(\F)|^r = |\F|^{3r(e-1)}  \cdot (|\F|^3 - |\F|)^r, \\
 = & \ |\F|^{3re} \cdot \left(1 - \frac{1}{|\F|^2}\right)^{r} \\
 = & \ 2^{3d} \left(1 - \frac{1}{2^{2f}}\right)^{r}. \end{aligned}
$$
An explicit table of the indices for~$3 \le n \le 16$ is given below:
\medskip
\begin{center}
\begin{tabular}{|c|c|c|c|c|c|c|c|c|c|c|c|c|c|c|}
\hline
$n$ & $3$ & $4$  & $5$  & $6$  & $7$  & $8$  & $9$ & $10$ & $11$ & $12$& $13$& $14$& $15$& $16$
\\
\hline
$[\wGamma:\wDelta] \rule{0pt}{2.5ex}$ &
$1$ &
$3$ &
$72$ &
$8$ &
$144$ &
$24$ &
$112$ &
$288$ &
$23808$ &
$8$ &
$322560$ &
$6912$ &
$6528$ &
$3072$  \\
\hline
\end{tabular}
\end{center}
\end{remark}


\subsection{Acknowledgments}

I would like to thank  Curt McMullen both for encouraging me to write this note
and also for  a number of useful discussions, remarks, and references.

\section{Some preliminaries}

In this section, we record some preliminary facts and lemmas which will be used in the sequel.
In addition to the field~$K = \Q(\zeta+\zeta^{-1})$, we shall also consider~$L = \Q(\zeta)$.
If we write~$n = 2^k m$ with~$m$ odd, then~$\Gal(L/\Q) \simeq (\Z/n \Z)^{\times}
= (\Z/m \Z)^{\times} \times (\Z/2^k \Z)^{\times}$, whereas~$\Gal(K/\Q)$ is the quotient of this group
by the image of complex conjugation, given explicitly by~$c = (-1,-1)$. If~$m > 1$, it follows from this description
and the natural identification of the inertia group at~$2$ of~$\Gal(L/\Q)$ with~$(\Z/2^k \Z)^{\times}$
that~$L/K$ is unramified at all primes above~$2$. 
The ramification index~$e$ of (both) fields at the prime~$2$ is~$e = \varphi(2^k)$
which equals~$2^{k-1}$ if~$k \ge 1$ and~$1$ otherwise,
unless~$n = 2^k$ in which case~$e = 2^{k-2}$ for~$K$.
 Although~$\Delta$ is naturally a subgroup of~$\SL_2(\OL_K)$,
it will sometimes be convenient to work in the larger group~$\SL_2(\OL_L)$ where certain elements are naturally
diagonalizable, making a number of explicit computations easier.

\subsection{The adjoint representation}
Let~$J$ be an ideal of~$\OL_K$. The ideal~$J$ is principal in the ring~$\OL/J^2$
and any choice of generator gives an isomorphism~$J/J^2 \simeq \OL/J$.
 If~$\Gamma = \SL_2(\OL_K)$, then
 (with $\Gamma(J)$ defined as in Definition~\ref{defJ})
 there is a corresponding identification 
 \begin{equation}
 \label{firstadjoint}
 \Gamma(J)/\Gamma(J^2) \simeq I + M^0(J/J^2) \simeq  M^0(\OL/J),
 \end{equation}
  where~$M^0(A)$ denotes the~$2 \times 2$ matrices in~$A$ with trace zero (the trace
 zero condition comes from the fact that elements in~$\Gamma$ have determinant one).
 Not only is~(\ref{firstadjoint}) an isomorphism of (abelian) groups, but there is also a natural
 action of~$\Gamma/\Gamma(J) \simeq \SL_2(\OL/J)$ on~$\Gamma(J)/\Gamma(J^2)$
 such that the corresponding action on~$M^0(\OL/J)$ is given by conjugation.
 More generally, for any non-zero ideal~$P \subset \OL$,  we have
 $$\Gamma(JP)/\Gamma(J^2P) \simeq M^0(\OL/J)$$
 A natural approach to understanding the closure of a subgroup~$G$ of~$\SL_2(\OL_{\p})$ is 
  to consider the images of $G$ in $\SL_2(\OL/\p^r)$ for each $r$, or equivalently the
 filtration of~$G$ by~$G(\p^r)$. One then wants to  understand the  graded pieces~$G(\p^r)/G(\p^{r+1})$ as
 modules for the image~$G/G(\p)$ of~$G$ in~$\SL_2(\OL/\p)$.
 
 \begin{remark}
 A basic principle  which we apply repeatedly through the paper (sometimes explicitly
 but more often implicitly) is the following. Let~$B$ be a $p$-group
or a pro-$p$ group 
 and let~$A \subset B$ be a closed subgroup.
 In order to prove that~$A=B$, it suffices,
 by Burnside's
basis theorem~\cite[Thm~12.2.1]{Marshall}, to show that~$A$ surjects onto the maximal
exponent~$p$-abelian quotient~$B/\Phi(B)$ of~$B$.
In the context of our problem, one should imagine taking~$A$
to be~$\wDelta(\p)$, and~$B \subset \wGamma(\p)$ to be the subgroup which we hope
to show is equal to~$A$. One measure of the difficulty in applying Burnside's
theorem is the size of~$B/\Phi(B)$. For the particular
problem we are considering,   these quotients
are  quite small when~$p > 2$, since in all cases we can
take~$B = \wGamma(\p)$  with~$\Phi(B) = \wGamma(\p^2)$,
However, when~$p=2$, $B$ has large index in~$\wGamma(\p)$,  and
the Frattini quotient~$B/\Phi(B)$ is also large; this is reflected
by the fact that we need to work harder to pin
down the precise image when~$p=2$.
\end{remark}

 \subsection{The adjoint representation at~$2$}
It will be of particular importance when studying the~$2$-adic image of~$\Delta$ to understand
the basic group theory of these modules, which we discuss now.
If~$n = 2^k m$ with~$m > 1$ odd, it is shown in~\cite[Theorem~1.1]{CM} that the image of~$\Delta$
in~$\SL_2(\OL_K/2)$ is isomorphic to~$D_{n/2}$ if~$k > 0$ and~$D_n$ otherwise.
If~$\p$ is a prime of residue characteristic two, then the image of~$\Delta$ in~$\SL_2(\OL_K/\p)$
is isomorphic to the (typically smaller) group~$D_m$.  (The importance of the image of~$\Delta$ in~$\SL_2(\OL_K/\p)$ is why
we reserve~$m$ for the largest odd divisor of~$n$, in contrast to~\cite{CM} where~$n=m$ or~$n = 2m$
depending on whether~$n$ is odd or not; hopefully this conflict does not cause any confusion.)

Let~$n = 2^k m$,
Let~$\zeta = \zeta_{n}$, and let
$$T = \left( \begin{matrix} 1 & 2 + \zeta + \zeta^{-1}  \\ 0 & 1 \end{matrix} \right),
\quad
U = - \left( \begin{matrix} 1 & 0 \\ -1 & 1 \end{matrix} \right),
\quad
R = UT.$$
The matrix~$R$ has order~$n$ and eigenvalues~$\zeta$ and~$\zeta^{-1}$.

For some purposes, it will
be useful to note that  we can explicitly diagonalize~$R$ by the matrix of eigenvectors:
\begin{equation}
\label{Qmat}
\QQ = \left( \begin{matrix}
-1 - \zeta & -1 - \zeta^{-1} \\
1 & 1 \end{matrix} \right).
\end{equation}
Note that this matrix lies in~$\GL_2(\OL_{\p}(\zeta))$ for any prime~$\p | 2$ only if~$m > 1$ (so that~$1+\zeta$
is a~$2$-adic unit).
In the new basis, we have
\begin{equation} \label{nicebasis}
\begin{aligned}
\QQ^{-1} R \QQ = & \ 
 \left( \begin{matrix} \zeta^{-1} & 0 \\ 0 & \zeta \end{matrix} \right),\\
\QQ^{-1} T \QQ = & \   \frac{1}{1-\zeta} \left( \begin{matrix}  2 & 1 + \zeta \\ -1 - \zeta & - 2 \zeta \end{matrix}\right), \\
\QQ^{-1} U  \QQ= & \ \frac{1}{1-\zeta} \left(\begin{matrix} -2 & -1 - \zeta^{-1} \\
\zeta(1+\zeta) & 2 \zeta \end{matrix}\right). \end{aligned}
\end{equation}
The basis~(\ref{nicebasis}) will occasionally be more convenient
for some computations. Note that it is only defined over the larger field~$L = \Q(\zeta)$
rather than~$K = \Q(\zeta+ \zeta^{-1})$. We only use this basis when~$m > 1$.

Let us now assume that~$m > 1$.
(The case~$n=2^k$ is considered in Section~\ref{poweroftwo}.)
 Let~$D_m$ denote the dihedral group of order~$2m$.
Let~$\F$ be the finite field of characteristic~$2$ given by adjoining~$\zeta + \zeta^{-1}$
to~$\F_2$
for a primitive~$m$th root of unity~$\zeta$. (Note that if~$\zeta$ is a primitive~$2^k m$th root of unity,
then the image of~$\zeta$ in any field of characteristic~$2$ is a primitive~$m$th root of unity.)
There is an isomorphism $\OL/\p = \F$ for any prime $\p | 2$.
The group~$D_m$ admits a~$2$-dimensional representation over~$\F$ given by the image of~$\Delta$
in~$\SL_2(\OL/\p)$ for a prime~$\p | 2$. 
The action of conjugation makes~$M^0(\F)$ a representation for~$D_m$. 
There is a splitting
$$M^0(\F) \simeq \F \oplus Q$$
 of this representation as a~$D_m$ module given explicitly by
\begin{equation}
\label{moduleold}
\left(\begin{matrix} a & 0 \\ 0 & a \end{matrix} \right) \oplus \left(\begin{matrix} b+c & b(\zeta + \zeta^{-1})
 \\ c & b+c \end{matrix} \right) 
\end{equation}

The representation~$Q/\F$ is absolutely irreducible. However, more is true:

\begin{lemma} \label{isirreducible} The representation~$Q$ considered as a representation of~$D_m$ over~$\F_2$
is irreducible.
\end{lemma}

\begin{proof}
Over~$\F_2$, the representation~$Q/\F_2$ has dimension~$2[\F:\F_2]$.
The base change~$Q/\F_2 \otimes_{\F_2} \F$
decomposes as a product of all the~$\Gal(\F/\F_2)$-conjugates of~$Q$ which are distinct
 representations because the trace of an order~$m$ element generates~$\F$. It follows that the only~$\Gal(\F/\F_2)$-invariant subrepresentation is the
 entire space, and thus~$Q/\F_2$ is irreducible.
 \end{proof}
 
 Now suppose that~$P = \prod_{\p | 2} \p$. The image of~$\Delta$ in~$\SL_2(\OL/P)$ is once
 again isomorphic to~$D_m$.  If~$r$ is any integer, there is an isomorphism
 $$\Gamma(P^r)/\Gamma(P^{r+1}) \simeq  \bigoplus_{\p|2} \Gamma(P^r)/\Gamma(P^r \p) \simeq \bigoplus_{\p|2} M^0(\F)
 \simeq \bigoplus_{\p|2} (\F \oplus Q).$$
 Our goal will be to study the filtration~$\Delta(P^r)/\Delta(P^{r+1})$ by showing
 that the image in~$\Gamma(P^r)/\Gamma(P^{r+1})$ contains the~$Q$-isotypic component of~$\Gamma(P^r)/\Gamma(P^{r+1})$
 for all~$r \ge e$. 
 
 It will also be useful to consider~$\Gamma(2)/\Gamma(4)$ as a module for the image~$D_{n/2}$ 
 (or~$D_n$ if~$n$ is odd)
 inside~$\SL_2(\OL/2)$. There certainly exists a splitting:
 $$\Gamma(2)/\Gamma(4) \simeq I + 2 M^0(\OL/2) \simeq M^0(\OL/2)  \simeq (\OL/2) \oplus \wQ,$$
 where~$\OL/2$ is the subgroup consisting of scalar matrices, and~$\wQ \simeq (\OL/2)^2$ is the group given by
 matrices of the form
 \begin{equation} \label{moduleoldq}
 \left(\begin{matrix} b+c & b(\zeta + \zeta^{-1})
 \\ c & b+c \end{matrix} \right) 
\end{equation}
just as in the decomposition~(\ref{moduleold}) above (remember we are assuming
that~$m > 1$ and so~$\zeta + \zeta^{-1}$ is a $2$-adic  unit). All we shall use about~$\wQ$ is that it is stable under conjugation
by the dihedral group.

\subsection{Roots of unity}

We have the following elementary construction involving roots of unity.

\begin{lemma} \label{subspacenew}
Let~$\zeta$ be a root of unity of order~$2^k m$
with~$m > 1$ odd. Let~$r \ge 0$ be an integer such that~$2r+1 < e$, where~$e$
is the ramification index of~$2$ in~$L$.
Let~$\xi = \zeta^m = \zeta_{2^k}$. 
Let~$\OL_L = \Z[\zeta]$, so~$(1 - \xi) \OL_L = \prod \P_i$, where~$\P_i$ are the primes
above~$2$ in~$\OL$.
The~$\Z$-span of the elements
$$\zeta^i - \zeta^j = \zeta^i + \zeta^j \mod 2$$
where~$i$ and~$j$ are odd contains an element of~$\OL_L/2$ with valuation exactly~$2r$
at~$\P$ and valuation greater than~$2r$ at all other primes above~$2$.
The~$\Z$-span of the elements
$$\zeta^i(\xi^j - 1) \mod 2$$
where~$i$ and~$j$ are odd contains an element of~$\OL_L/2$ with valuation~$2r+1$
at~$\P$ and valuation greater than~$2r+1$ at all other primes above~$2$.
\end{lemma}

\begin{proof}
Since~$\xi-1$ has valuation~$1$ at all primes above $2$, the second claim follows from the first (even taking~$j = 1$).
If~$\alpha \in \OL_L$ is expressible as a sum of an even number of roots of unity
to odd exponent, then so is
$$\alpha (\xi^2 - 1)  = \alpha \zeta^{2m} - \alpha,$$
and the valuation of~$\alpha (\xi^2 - 1) = \alpha (\xi - 1)^2 \bmod 2$ at any~$\P$ is exactly $2$
plus the valuation of~$\alpha$. Hence, by induction,  it suffices to consider the case when~$r = 0$.
Certainly there exists an~$\eta \in \OL$ which is divisible by all primes above~$2$ with the exception of
a chosen prime~$\P$. 
Multiplying by the unit~$(\zeta - 1)$ (which is a unit since~$m > 1$), we can assume that~$\eta$ is expressible
as a sum of an even number of roots of unity. Replacing any even odd exponent~$\zeta$
by~$\zeta^{i+m}$ keeps~$\eta$ unchanged modulo~$(1-\xi)\OL_L = \prod \P_i$, which preserves the property
of being co-prime to~$\P$ but not having positive valuation at every other prime above~$2$, and hence we have constructed
 the desired element.
\end{proof}

\section{The~$2$-adic image}
\label{2adicfirst}

We now consider the~$2$-adic image of~$\Delta$ when~$n = 2^k m$ and~$m > 1$ is odd, deferring the case of~$n = 2^k$
to Section~\ref{poweroftwo}.
To prove Theorem~\ref{2part}, we both need to show that~$\wDelta_2$ is both contained
and equal to what is described there.  We begin by proving the desired containment.
 We already know the image of~$\Delta$ in~$\SL_2(\OL/2)$
from~\cite[Thm~1.1]{CM}, so we only need consider the image of~$\Delta(2)$ in~$\Gamma(2)/\Gamma(4)$.
The statement to be proved is that the image of~$\Gamma(2)$ is contained within the subgroup
$$\F_2 \oplus \wQ \subset \OL/2 \oplus \wQ \simeq \Gamma(2)/\Gamma(4),$$
where the first~$\F_2$ is generated by the image of the matrices~$\pm I$.
(As noted in~\cite[\S1.V]{CM}, the choice of signs in the definition of~$\Delta$
ensures that~$-I \in \Delta$.)
Because~$\Delta(2)$ is normal in~$\Delta$, and since this subspace is invariant under conjugation by~$\Delta/\Delta(2)$,
it suffices to show that a set of \emph{normal} generators for~$\Delta(2)$ maps to this space.
But we know that, with~$d = n/2$ if~$n$ is even and~$n$ otherwise, that
$$\Delta/\Delta(2) \simeq \langle T,U | T^2, U^2, R^{d}, R =  UT \rangle.$$
So it suffices to check the claim for~$T^2$, $U^2$, and~$R^{d}$. The image of the latter is given by~$-I$.
On the other hand, we find that (modulo~$4$):
$$U^2 = \left( \begin{matrix} 1 & 0 \\ 2 & 1 \end{matrix} \right)
\equiv 1 + 2 \left(\begin{matrix} a & 0 \\ 0 & a \end{matrix} \right) + 2  \left(\begin{matrix} b+c & b(\zeta + \zeta^{-1})
 \\ c & b+c \end{matrix} \right)  \bmod 4 $$
 with~$(a,b,c) = (1,0,1)$,  whereas
 $$T^2 = \left( \begin{matrix} 1 &  4 + 2(\zeta+ \zeta^{-1}) \\ 0 & 1 \end{matrix} \right)
\equiv 1 + 2 \left(\begin{matrix} a & 0 \\ 0 & a \end{matrix} \right)+ 2  \left(\begin{matrix} b+c & b(\zeta + \zeta^{-1})
 \\ c & b+c \end{matrix} \right)   \bmod 4$$
 with~$(a,b,c) = (1,1,0)$.
 
We now turn to showing that~$\wDelta_2$ is as big as possible.
We reduce it to the following:
\begin{lemma} \label{satisfied}
Let~$e$ be the ramification degree of~$2$ in~$K$, and let~$P = \prod_{\p|2} \p$.
Suppose there are~$d$ primes dividing~$2$ in~$\OL_K$.
Suppose that the image of~$\Delta(P^r)/\Delta(P^{r+1})$ inside
 $$\Gamma(P^r)/\Gamma(P^{r+1}) \simeq  \bigoplus_{\p|2} \Gamma(P^r)/\Gamma(P^r \p) \simeq \bigoplus_{\p|2} M^0(\F)
 \simeq \bigoplus_{\p|2} (\F \oplus Q)$$
 contains all~$d$ copies of~$Q$ for every~$r=e, e+1, \ldots, 2e-1$.
Then Theorem~\ref{2part} holds.
 \end{lemma}

\begin{proof}
Since we have already seen that the image of~$\Delta(2)/\Delta(4)$ in~$\OL/2 \oplus \wQ$ lands
in~$\F_2 \oplus \wQ$ and contains the copy of~$\F_2$ generated by~$-I$, this certainly shows that the image is the entire
space~$\F_2 \oplus \wQ$. Thus it suffices to prove that~$\wDelta_2(4) = \wGamma_2(4)$, or that
the image of~$\Delta(2^r)$ surjects onto~$\Gamma(2^r)/\Gamma(2^{r+1})$ for every~$r \ge 2$.
We prove this by induction. 
 Note that
$$\Delta(2^r)/\Delta(2^{r+1}) \simeq \OL/2 \oplus \wQ.$$
Suppose either that~$r \ge 2$ and that the image  is surjective, or that~$r=1$ and (by assumption)
the image contains~$\wQ$. In both cases, the image
 certainly contains~$\wQ$.
Hence for  any  pair of elements in~$M_0(\OL/2)$, there
exist elements
$$I +  2^r A \bmod 2^{r+1}, \quad 1 + 2 B \bmod 2^2$$
in the image of~$\Delta$ where~$A$ and~$B$ realize the given elements up to a scalar matrix. 
Hence the image of~$\Delta$ contains the commutator
$$[I + 2^r A \bmod 2^{r+1},I + 2 B \bmod 4] = I + 2^{r+1} (AB - BA) \bmod 2^{r+2}$$
Now replacing~$A$ and~$B$ by~$A + \lambda I$ and~$B + \mu I$ does't change~$AB-BA$.
On the other hand, the image of 
the map
$$M^0(\OL/2)^2 \rightarrow M^0(\OL/2)$$
given by~$(A,B) \rightarrow AB-BA$ is the full space~$\OL/2$ of scalar matrices.
Hence the image of~$\Delta(2^r)/\Delta(2^{r+1})$ certainly contains this subgroup~$\OL/2$ for~$r \ge 2$.
With~$A$ as above, we deduce that~$\Delta(2^r)/\Delta(2^{r+1})$
also contains the image of
$$(I + 2^r A)^2 = I + 2^{r+1} A + 2^{2r} A^2 \mod 2^{r+2}.$$
If~$r > 1$, this is~$I + 2^{r+1} A \bmod 2^{r+2}$, 
and since we can choose~$A$ to be any element in~$M^0(\OL/2)$ up to scalars, and we have also showed that the image contains
all the scalars, we are done by induction. In the case when~$r = 1$, this is~$I + 4(A + A^2) \bmod 8$.
But now since~$A \in M^0(\OL/2)$ has trace zero it follows that~$A^2$ is scalar, so just as above
 we can realize~$I + 4A \mod 8$ in the image of~$\Delta(4)/\Delta(8)$ 
for any~$A \in M^0(\OL/2)$, and we are done.
\end{proof}

\begin{proof}[Proof of Theorem~\ref{2part}]
To prove Theorem~\ref{2part}, it suffices to show that the conditions of Lemma~\ref{satisfied} are satisfied.
 For that, we simply write down some explicit elements
first in~$\Delta(2) = \Delta(P^e)$ and then in~$\Delta(P^r)$ for~$r \in [e,\ldots,2e-1]$ and compute their images
in~$\Delta(P^r)/\Delta(P^{r+1})$. 
We have, for example, the elements
$$X_{i,j}:=(R^i T)^2 (R^j T)^{-2}
= I + \frac{2}{\zeta - 1} \QQ \left( \begin{matrix} 0 & \zeta(\zeta^{-2i-1} - \zeta^{-2j-1}) \\
 \zeta^{2i+1} - \zeta^{2j+1}  & 0\end{matrix} \right) \QQ^{-1} \mod 4,$$
 where~$\QQ$ is the matrix~(\ref{Qmat}) which we remind the reader 
 is~$2$-adically invertible  since the determinant~$\zeta - \zeta^{-1}$ is a unit when~$m > 1$.

Let us define
\begin{equation}
\label{defVi}
V_{i}:=R^{mi/2+ 1/2}T^2 R^{-mi/2 - 1/2} U^{-2} \in \Delta
\end{equation}
for any odd integer~$i$ (since~$m$ is odd, this guarantees that the exponent of~$R$
is an integer) and then we find:
\begin{equation}
\label{defW}
W_{i,j}:=R^{(j-1)/2} V_i R^{-(j-1)/2}
\equiv
I +  2  
\QQ \left( \begin{matrix}
0 &
 \zeta^{-j}(\xi^{-i} - 1) \\
 \zeta^{j} (\xi^i - 1) &
0 \end{matrix} \right) \QQ^{-1} \bmod 4
\end{equation}
now with~$i$ and~$j$ both odd.
For both~$X_{i,j}$ and $W_{i,j}$, note that the off-diagonal matrices are swapped
by the automorphism~$c$ sending~$\zeta$ to~$\zeta^{-1}$.
If~$M_i \in M_2(\OL_L)$ 
is any finite set of matrices, there is the elementary congruence:
\begin{equation}
\label{elementary}
\prod (I + 2 M_i) = I + 2 \sum M_i \bmod 4.
\end{equation}
It follows that, by either taking products of~$X_{i,j}$ or~$W_{i,j}$,
we can find elements in~$\Delta$ of the form
$$A = I + \QQ B \QQ^{-1} \bmod 4$$
where~$B$ is zero on the diagonal and on the lower left hand corner is given (up to a unit scalar)
either by a sum of
$$\zeta^i - \zeta^j$$
where~$i$ and~$j$ are both odd, or
$$\zeta^j (\xi^i - 1)$$
with~$i$ and~$j$ both odd.

Now let~$\p | 2$ in~$\OL_K$, and let~$\P$ be a prime in~$\OL_L$ above~$\p$.
Note that there are at most two primes above~$\p$ in~$\OL_L$, and the other has
the form~$c \P$ where~$c$ is given by complex conjugation.
By Lemma~\ref{subspacenew}, we may, depending on the parity of~$r$,
find an~$B$ such that the lower left corner has any given valuation at~$\P$ and higher valuation at all other primes.
Moreover, since the upper right corner is obtained by acting by complex conjugation, it has the same valuation
at~$c \P$ and higher valuation at all other primes. Reverting back to the basis over~$\OL_K$ 
and to the matrix~$A$, 
the fact that~$E$ is~$2$-adically invertible means that since all the entries in the original matrix were
trivial modulo~$P^r \p'$ for every prime~$\p'$ except~$\p$ means 
 that~$A \in \Gamma(P^r \cdot P/\p)$ and moreover the image
of~$A$ in~$\Gamma(P^{r})/\Gamma(P^{r+1})$  is non-zero. Moreover,
it is non-scalar as well, since otherwise (conjugating again by~$E$) it would still be scalar,
and by construction it is not of this form.
 It follows that its image
in this quotient lies in precisely one copy
of~$Q$ associated to the factor~$\Gamma(P^r)/\Gamma(P^r \p)$. 
Hence, because~$Q$ is irreducible over~$\F_2$, we deduce that the image of~$\Gamma(P^r)$
contains all~$d$ copies of~$Q$, as desired.
\end{proof}

\section{The adelic image}

We now turn to computing the adelic image away from the prime~$p=2$. In this section,
we now allow~$n = 2^k$.
The groups~$\SL_2(\F)$ have no normal subgroups except for the center when~$|\F| \ge 5$.
Moreover~$\SL_2(\F)$ and~$\SL_2(\F')$
have no isomorphic non-trivial quotients when they are distinct fields of non-equal
characteristic.

\begin{lemma} \label{dickson}
Let~$p$ be odd.
Suppose that~$\F_p[\zeta+\zeta^{-1}] = \F$. Then~$U$ and~$T$ generate~$\SL_2(\F)$
when~$\F$ has odd characteristic with the following exceptions:
\begin{enumerate}
\item $n = 5 \cdot 3^k$ for some~$k \ge 0$ and~$p = 3$, in which case  the image inside~$\SL_2(\F_9)$ for a prime above~$3$ is isomorphic
to the degree~$2$ central extension of~$A_5 \subset A_6 \simeq \PSL_2(\F_9)$. 
\item $n = 2p^k$ for some~$k \ge 1$, in which case the image is cyclic of order~$p$ generated by the image of~$U$.
\end{enumerate}
Moreover, $\PSL_2(\F)$ is simple unless~$n \in \{3^k, 2 \cdot 3^k, 4 \cdot 3^k\}$,
in which case the image is cyclic of order~$3$ or~$\SL_2(\F_3)$ depending on whether~$n$
is twice a prime power or not.
\end{lemma}

\begin{proof}
Note that the image of~$U$ and~$T$ are both unipotent and non-trivial
unless~$2 + \zeta + \zeta^{-1} \equiv 0 \bmod \p$. The latter implies that~$\zeta \equiv -1 \bmod \p$,
and thus that~$\zeta$ has order exactly~$2$ in~$\F$. But if~$n = m \cdot p^k$ where~$(m,p)=1$,
then~$\zeta$ has exact order~$m$, so~$m=2$. Hence we may assume that~$2+\zeta+\zeta^{-1}$
is non-zero and generates~$\F$.

Any two such unipotent elements generate~$\SL_2(\F)$ by a theorem of 
Dickson~\cite[Thm 8.4, p.44]{Gorenstein} 
unless~$|\F| = 9$, in which case one can compute directly that the same claim holds unless~$\eta =2 + \zeta + \zeta^{-1}$ satisfies~$\eta^2 + 1 = 0$, in which case the
image has order~$120$ and maps to~$A_5 \subset A_6 = \PSL_2(\F_9)$.
The condition that~$\eta^2+1=0$ is equivalent to
$$0 = \zeta^2 + \zeta + 1 + \zeta^{-1} + \zeta^{-2},$$
or equivalently~$\zeta^5-1=0$. This  happens if and only if~$\zeta$ has exact order~$5$ in characteristic~$3$,
which implies the first result.

The order of the residue field at~$3$ is, writing~$n = m \cdot 3^k$ with~$(3,m)=1$,
the order of~$3$ in~$(\Z/m\Z)^{\times}/\pm 1$. This has trivial order if and only if~$3 \equiv \pm 1 \bmod m$,
or when~$m \in \{1,2,4\}$.
\end{proof}

We now prove the following:

\begin{lemma} \label{pbyp}
The closure of the adelic image of~$\Delta$ is the product of the closures
of the images in~$\SL_2(\OL \otimes \Z_p)$ for all~$p$.
Moreover, if~$n \ne 2 p^k$ for some~$k \ge 1$, then the image in~$\SL_2(\OL \otimes \Z_p)$
for~$p$ odd
is the direct product of the images inside~$\SL_2(\OL_{\p})$ for all~$\p | p$.
\end{lemma}

\begin{proof} By Goursat's Lemma, this is automatically
true unless the different images have common isomorphic finite quotients.
For primes above~$2$, the image is of dihedral type.
For primes above~$3$, the image is either pro-$3$ or a quotient of~$\SL_2(\F)$
for some~$\F$ of residue characteristic~$3$, and any quotient of~$\SL_2(\F_3)$
surjects onto~$\Z/3\Z$ whereas no quotient of the dihedral group has this property.
Other than that, there are no coincidences between~$\SL_2(\F)$
(or~$2.A_5$) between fields of different characteristic. This proves the first claim.
For the second, our assumption implies that the image in~$\SL_2(\F)$ is
either~$\SL_2(\F)$ or the one possible exception inside~$\SL_2(\F_9)$.
If the image is~$\SL_2(\F_3)$, then~$n = 3^k$, $2 \cdot 3^k$, or~$4 \cdot 3^k$; in each case
there is a unique prime above~$3$ so there is nothing to prove.
This also happens when~$n = 5 \cdot 3^k$ for primes in~$\Z[\zeta+\zeta^{-1}]$
since~$3$ is inert in~$\Q(\sqrt{5})$. Thus we may assume that the residual
images of all the~$p$-adic representations is equal to~$\SL_2(\F)$.

By Goursat's lemma, it suffices to prove that the map to
$$\prod_{\p} \SL_2(\OL/\p)$$
is also surjective. 
Assume first that~$n = m p^k$. Since there is only something to prove when there are at least
two primes above~$p$ in~$K$, we may also assume that~$m \ge 3$.
Write~$\eta = \zeta^{p^k}$ and let~$F \subset K$ be the field~$F = \Q(\eta + \eta^{-1})$.
Then~$K/F$ is totally ramified at all primes above~$p$, we have~$\OL_F
= \Z[\eta + \eta^{-1}]$. For each prime~$\p$ in~$\OL$ above~$p$,
if~$\p_F = \p \cap \OL_F$ then~$\OL_F/\p_F = \OL/\p$. In particular,
there is a canonical isomorphism of target groups
\begin{equation}
\label{theyarethesame}
\prod_{\p | p} \SL_2(\OL/\p) \simeq \prod_{\p_F | p}  \SL_2(\OL_F/\p_F).
\end{equation}
Since~$\zeta^{p^n} \equiv \eta \bmod \p$,
 the images of~$\Delta_n$ and~$\Delta_m$ coincide (up to conjugation
 in~$\Gal(K/\Q)$ by the inverse of~$[p^k] \in (\Z/m \Z)^{\times}/\pm 1$)
after the targets are identified by equation~(\ref{theyarethesame}).
This reduces the claim for~$\Delta_n$ to the for~$\Delta_m$,
and thus we may assume that~$(n,p) = 1$.

By~\cite[Lemma~3.3]{Ribet}, it suffices to prove the image surjects onto any pair of factors
(note that we are in the situation where we may assume that~$|\F| \ge 5$ so~$\SL_2(\F)$ is perfect).
By Goursat's lemma, the result follows unless the image is a graph~$(x,\phi(x))$ for some isomorphism~$\phi: \SL_2(\F) \rightarrow \SL_2(\F)$.
 It follows by considering the image of~$U$ that~$\phi$ fixes
$$\left( \begin{matrix} 1 & 0 \\ 2 & 1 \end{matrix} \right).$$
Let us examine the image of~$T$ in~$\SL_2(\OL/\p)$ and~$\SL_2(\OL/\sigma \p)$
respectively
for a pair of primes~$\sigma \p \ne \p$.
We fix an isomorphism~$\F = \OL/\p$, so that the image of~$\zeta + \zeta^{-1}$
in~$\OL/\p$  is~$\zeta+\zeta^{-1}$. The image of~$\zeta+\zeta^{-1}$
in~$\F= \OL/\sigma \p$
has the form~$\sigma^{-1}(\zeta+\zeta^{-1})$ in~$\OL/\p$.
Thus
\begin{equation} \label{goursat}
\begin{aligned}
 \phi \left(\left( \begin{matrix} 1 & 0 \\ 2 & 1 \end{matrix} \right) \right) &  = \left( \begin{matrix} 1 & 0 \\ 2 & 1 \end{matrix} \right), \\
\phi \left(\left( \begin{matrix} 1 & 2 + \zeta + \zeta^{-1} \\
0 & 1 \end{matrix} \right) \right)
& = 
\left( \begin{matrix} 1 & 2 +\sigma^{-1}(\zeta + \zeta^{-1}) \\
0 & 1 \end{matrix} \right). \end{aligned}
\end{equation}
The assumption that~$\p \ne \sigma \p$
is the assumption that~$\sigma$ is not in the subgroup of~$\Gal(K/\Q)$
 generated by the Frobenius
element~$\Frob_p \in \Gal(K/\Q)$ (which would mean that~$\p = \sigma \p$).
All automorphisms of~$\SL_2(\F)$ can be written as a product
of a diagonal automorphism (conjugation by~$\GL_2(\F)$)
and the action of~$\Gal(\F/\F_p)$, which is generated by Frobenius. 
(The latter are called field automorphisms or Galois automorphisms.)
The centralizer of a unipotent element in~$\GL_2(\F)$ meanwhile consists of the group
generated by that element together with scalar matrices (and the latter give trivial
automorphisms.)
In particular, the only $\phi$ compatible with the first line of~\eqref{goursat}
is a product of a field automorphism with conjugation by some power of the given element.
The only such automorphisms which send a (non-diagonal) upper triangular unipotent
matrix to another upper triangular unipotent matrix are the field automorphisms.
Thus~$\phi$ has to be a Galois
automorphism, and thus it must be the case that:
$$\Frob^i_p (\zeta +\zeta^{-1}) \equiv \sigma^{-1}(\zeta + \zeta^{-1}) \mod \p,$$
for some~$i$ where~$\sigma \in \Gal(K/\Q)$ is not in the group generated by~$\Frob_p$. 
In particular, these are different 
global conjugates of $\zeta + \zeta^{-1}$ and hence such a congruence is impossible
if the minimal polynomial of~$\zeta+\zeta^{-1}$ is separated modulo~$p$.
But the ring of integers in~$\Q(\zeta+\zeta^{-1})$ is~$\Z[\zeta+\zeta^{-1}]$,
and since we are assuming that~$\zeta$ has order~$(m,p)=1$, this field is unramified at~$p$,
and thus we are done.
\end{proof}

We now have:

\begin{lemma} \label{unramp}
Suppose that~$p$ is odd and~$(n,p)=1$. Then the map~$\Delta \rightarrow \SL_2(\OL_{\p})$
is surjective.
\end{lemma}

\begin{proof}
We have seen that the image to~$\SL_2(\F)$
is surjective.  Since~$p$ is unramified, the
image of~$T^{p^r}$ is non-trivial in~$\Delta(\p^r)/\Delta(\p^{r+1})$
for every~$r \ge 1$. If the adjoint representation is irreducible as a representation
of~$\SL_2(\F)$, then the image must be all of~$\Delta(\p^r)/\Delta(\p^{r+1})$,
and then the result follows by induction. To see that the adjoint representation~$Q := 
\Delta(\p)/\Delta(\p^2)$
is irreducible as an~$\F_p[\SL_2(\F)]$-module, we may
argue as in Lemma~\ref{isirreducible}. Namely, $Q \otimes_{\F_p} \F$ decomposes
as a direct sum of the~$\Gal(\F/\F_p)$-conjugates of the adjoint representation over~$\F$,
and as long as the traces of elements of~$\SL_2(\F)$ on this module generate~$\F$,
the conjugates are all distinct. Any element~$\gamma \in \F$ is the trace of a matrix in~$\SL_2(\F)$,
and the trace of the same element on the adjoint representation is~$\gamma^2 - 1$. So it suffices
to observe that the elements~$\gamma^2$ for~$\gamma \in \F$ generate~$\F$ over~$\F_p$.
\end{proof}

\subsection{Ramified Primes} We now deal with (odd) ramified primes.
As we have seen, if~$n = 2 p^k$ and~$k \ge 1$, then the mod-$\p$ representations
themselves only have cyclic image, so we defer this case until
Section~\ref{redcase}.
We begin with some preliminaries. 
For any integer~$d \ge 2$,
we certainly have an exact sequence of groups:
\begin{equation}
\label{reducefrattini}
0 \rightarrow M \rightarrow \SL_2(\OL/\p^{d}) \rightarrow \SL_2(\F) \rightarrow 0.
\end{equation}
Letting~$\pi$ be a generator of~$\p/\p^2$, we may think of elements
of~$M$ as matrices of the form~$1 + \pi A$ with~$A \in M(\OL/\p^{d-1})$ (but
note that this does not respect the group structure). The group~$M = M(\p)$
has an obvious filtration~$M(\p^i)$ and~$M^0(\F) \simeq M(\p^i)/M(\p^{i+1})$
sending~$A$ to~$1 + \pi^i A$ which is an isomorphism both of abelian groups
and  as modules for~$\F_p[\SL_2(\F)]$.

We first need to
understand the structure of the~$p$-group~$M$.
Let~$\Phi(M) \subset M$ denote the Frattini subgroup (the intersection
of all maximal subgroups of~$M$).
Since~$M$ is a~$p$-group, it has the form~$\Phi(M)=[M,M]M^p$,
and~$M/\Phi(M)$ is elementary~$p$-abelian.

\begin{lemma} \label{frattini} The Frattini subgroup~$\Phi(M)$
is isomorphic to~$M(\p^2)$,
so the Frattini quotient~$M/\Phi(M)$ is isomorphic
to the image of~$M$ in~$\SL_2(\OL/\p^2)$. 
\end{lemma}

\begin{proof}
For~$i,j \in 1$, we have:
$$[I + \pi^i A,1 + \pi^j B] = 1 + \pi^{i+j}(AB - BA) \bmod \pi^{i+j+1}.$$
But since~$AB-BA$ generates~$M^0(\F)$ (assuming as we are that~$p > 2$), we see that~$[M,M]$
inductively contains~$M(\p^i)$
starting with~$i = d$ and descending down to~$i = 2$.
Hence~$\Phi(M) \supseteq [M,M]$ contains this subgroup as well. But~$M/M(\p^2)$
is elementary~$p$-abelian, so~$\Phi(M) \subseteq M(\p^2)$ and thus~$\Phi(M) = M(\p^2)$.
\end{proof}

An alternative way to phrase this is that the Frattini quotient of the pro-$p$
group~$M$ which is the kernel of~$\SL_2(\OL_{\p}) \rightarrow \SL_2(\F)$
has~$\Phi(M) \simeq M(\p^2)$, which corresponds to the limit case case~$d =\infty$.
We now have:

\begin{lemma} \label{mnot2} Let~$p$ be an odd prime dividing~$n$. Suppose that~$n$
is not of the form~$n = 2 p^k$ for~$k \ge 1$.
 Then the map~$\Delta \rightarrow \SL_2(\OL_{\p})$
is surjective, unless~$n = 5 \cdot 3^k$ for some~$k \ge 0$, in which case the image has index~$6$.
\end{lemma}

\begin{proof}
Let us consider the image~$G$ of~$\Delta$ in~$\SL_2(\OL/\p^d)$
for some integer~$d \ge 2$, and let~$G(\p)$ be the kernel
of the map~$G \rightarrow \SL_2(\F)$. 
We know, by assumption, that the image of~$G$ in~$\SL_2(\F)$ is either everything or has
index~$6$ in the one exceptional case. But in either case, the adjoint representation
is irreducible as an~$\F_p[\SL_2(\F)]$-module.

From Lemma~\ref{frattini}, there is  a map:
$$G(\p) \hookrightarrow M \rightarrow M/\Phi(M) \subset \SL_2(\OL/\p^2).$$
If~$G(\p)$ surjects onto~$M/\Phi(M)$ then~$G \simeq M$ by Burnside's
basis theorem~\cite[Thm~12.2.1]{Marshall}. Since~$d \ge 2$ was arbitrary,
the result follows.
 Hence it suffices to show
 that
 $$G(\p)/G(\p^2) \subseteq \Gamma(\p)/\Gamma(\p^2)$$
  is non-zero,
since the adjoint representation is irreducible.
In particular, it suffices to write down a single element with non-zero image.

Write~$n = m p^k$ with~$(m,p)=1$ and~$m \ne 2$.
Let~$v$ denote the valuation on~$K_{\p}$ so that~$v(\p \setminus \p^2) = 1$.
Note that~$v$ extends to any finite extension of~$K_{\p}$.
If~$\gamma \in G(\p^2)$, then
any eigenvalue~$\lambda \in \Kbar$  of~$\gamma$
satisfies
\begin{equation}
\label{lower}
v(\lambda - 1) \ge 2.
\end{equation}
Thus it suffices to write down an element~$\gamma$ which is trivial modulo~$\p$ (so~$\gamma \in G(\p)$)
 and has eigenvalues
which do not satisfy~(\ref{lower}) so~$\gamma \notin G(\p^2)$.
The element~$R$ has eigenvalues~$\zeta$ and~$\zeta^{-1}$.
If~$n = p^k m$ with~$m > 2$, then~$\zeta$ and~$\zeta^{-1}$
are distinct modulo~$\p$, and thus the reduction of~$R$ in~$\SL_2(\OL/\p)$
is semisimple. It follows that~$\gamma = R^m$
is trivial modulo~$\p$ and has eigenvalues~$\zeta^m = \zeta_{p^k}$ and~$\zeta^{-1}_{p^k}$.
If~$m > 2$, then~$\Q(\zeta)/\Q(\zeta + \zeta^{-1})$ is unramified at~$p$,
and thus~$v(\zeta_{p^k} - 1) = 1$. Thus~$\gamma$ gives the desired element,
and so we are done unless~$m \le 2$. We are assuming that~$m \ne 2$.
If~$m=1$,  then the reduction of~$R \bmod \p$ is not semisimple, and hence some alternate
argument is required, which we turn to now.

We now assume that $m=1$ so~$n =p^k$. There is now a unique prime~$\pi = \zeta + \zeta^{-1} - 2$
above~$p$ in~$K$ with residue field~$\F_p$.
We use  the following \emph{ad hoc} construction.
Let~$V = RU = UTU$; we will now show that~$\gamma = V^{p^2 - 1}$ lies in~$G(\p)$
and is non-trivial in~$G(\p)/G(\p^2)$. 
We first note that the mod~$\p$ reduction of~$V$ is as follows:
$$\Vbar = \left(\begin{matrix} -3 & 4 
\\ 2 & -3 \end{matrix} \right) \bmod \p.$$
The element~$\Vbar$ is semisimple 
because the characteristic polynomial~$x^2 + 6 x + 1 \in \OL/\p[x]$
is separable (the discriminant is~$32$ and~$p > 2$). Hence the order of~$\Vbar \in \SL_2(\F_p)$
divides~$p^2-1$, and so~$\gamma = V^{p^2-1}$ lies in~$G(\p)$.
Let~$\varepsilon = - 3 - 2 \sqrt{2}$ be a root of~$x^2+6x+1 \in K[x]$.
We note that~$\varepsilon^{p^2-1} \equiv 1 \bmod p$, so~$v(\varepsilon^{p^2 - 1} - 1) \ge e > 1$.
The eigenvalues of~$\Vbar$ may be identified
with the mod-$p$ reductions of~$\varepsilon$ and~$\varepsilon^{-1}$.
Let~$\eta$ be the  eigenvalue of~$V$ which is congruent to~$\varepsilon$ modulo~$\p$,
and let~$\lambda = \eta^{p^2 - 1}$ be the corresponding eigenvalue of~$\gamma$,
so certainly~$v(\lambda - 1) \ge 1$.
We claim  that if~$v(\eta/\varepsilon - 1) = 1$, then~$v(\lambda - 1) = 1$. Assume otherwise,
so~$v(\lambda - 1) > 1$.
If~$v(\xi-1) = 1$, then~$v(\xi^i - 1) = 1$ for~$(i,p) = 1$,
since~$(\xi^i-1)/(\xi-1) =1+\xi+\ldots \xi^{i-1} \equiv i \bmod \p$. Thus~$v(\eta/\varepsilon - 1) = 1$ implies that
$$
\begin{aligned}
1 = & \ v((\eta/\varepsilon)^{p^2-1} - 1) = v(\lambda/\varepsilon^{p^2-1} - 1) \\
= & \  v(\lambda - \varepsilon^{p^2-1})  - v(\varepsilon^{p^2 - 1})  =  v(\lambda - \varepsilon^{p^2-1})  \\
= & \  v\left( (\lambda - 1) + (1 - \varepsilon^{p^2 - 1})\right)\\
\ge & \ \min(v(\lambda - 1),v(1 - \varepsilon^{p^2 - 1}))\\
> & \ 1, \end{aligned}$$
since~$v(\varepsilon^{p^2 -1} - 1)  > 1$ and we are assuming that~$v(\lambda - 1) > 1$.
This is a contradiction.
Hence if~$v(\eta/\varepsilon - 1) = 1$, then~$v(\lambda-1) = 1$,
and then~$\gamma \in \Delta(\p)/\Delta(\p^2)$
is non-trivial, and we are done.
The characteristic polynomial of~$V$,
with~$\pi = \zeta+\zeta^{-1} - 2$,  is given by
$$T^2 + (6 + 2 \pi) T + 1.$$
We now compute that~$\eta/\varepsilon - 1$ is a root of the polynomial
$$T^2 + (6 \varepsilon^{-1} + 2 \varepsilon^{-1} \pi + 2) T + 2 \varepsilon^{-1} \pi = 0.$$
The Newton polygon of this polynomial has slopes~$0$ and~$1$, from which
 we deduce that~$v(\eta/\epsilon - 1) = 1$,
as desired. 
\end{proof}

\subsection{The $p$-adic image when~$n=2p^k$} \label{redcase}
We prove the following:

\begin{lemma} \label{ramp}
If~$n = 2 p^k$ with~$p > 2$ and~$k \ge 1$,
 the closure of~$\Delta$ in~$\SL_2(\OL \otimes \Z_p)$
 is a~$p$-Sylow subgroup, and the index coincides with the
 index of~$p$ in~$\SL_2(\F_p)$, namely~$p^2-1$.
\end{lemma}

\begin{proof}
When~$n=2p^k$
 there is a unique prime above~$p$ and the residue field is~$\F_p$.
 Our assumptions imply that~$\pi=2+\zeta+\zeta^{-1}$
is a uniformizer, and so the image of~$T$ in~$\SL_2(\OL/\p) = \SL_2(\F_p)$
is trivial.
For any~$d \ge 2$, we consider again the exact sequence~(\ref{reducefrattini}):
$$0 \rightarrow M \rightarrow \SL_2(\OL/\p^d) \rightarrow \SL_2(\F_p) \rightarrow 0.$$
By Lemma~\ref{frattini},  the Frattini quotient of~$M$ has rank~$3$ over~$\F_p$
 and is identified with
the image of~$M$ in~$\SL_2(\F_p[x]/x^2)$.
Let~$N$ be the preimage of~$\langle U \rangle$, so~$N$ contains~$M$ with index~$p$.
When~$n=2p^k$, the image~$G$ of~$\Delta$ in~$\SL_2(\OL/\p^d)$
 is a subgroup of~$N$, since~$U$ maps to~$N$ by construction and~$T$
 maps to~$M$.
We have the following elements in~$G \subset N$ 
 which are non-trivial in~$N/M(\p^2)$:
\begin{equation}
\begin{aligned}
V:=U^{(p+1)} \equiv & \  \left( \begin{matrix} 1 & 0  \\ 1 & 1 \end{matrix} \right) \bmod \pi^2, \\
T = &\ I + \pi  \left( \begin{matrix} 0 & 1  \\ 0 & 0 \end{matrix} \right), \\
VTV^{-1}\equiv  & \  I + \pi \left( \begin{matrix} -1 & 1  \\ -1 & 1 \end{matrix} \right)  \bmod \pi^2,\\
V^2TV^{-2}\equiv  & \  I + \pi \left( \begin{matrix} -2 & 1  \\ -4 & 2 \end{matrix} \right)  \bmod \pi^2, 
\end{aligned}
\end{equation}
but now assuming that~$p \ne 2$, as we may, we see that
the map~$G \cap M \rightarrow M/\Phi(M)$ is surjective,
and thus (once more) by
Burnside's
basis theorem, we deduce that~$G \cap M = M$,
and~$G = N$. Since~$d \ge 2$ was arbitrary, the result follows.
\end{proof}

\section{The~$2$-adic image when~$n = 2^k$} \label{poweroftwo}
To complete our analysis of the image, we need to understand the~$2$-adic closure~$\Delta_2$
of~$\Delta$ when~$n=2^k$ is a power of~$2$. 
The arguments will be quite similar to those in Section~\ref{2adicfirst} but sufficiently different to warrant
separate consideration.
Note that in this case~$\OL = \Z[\zeta+\zeta^{-1}]$ is totally
ramified of degree~$e = 2^{k-1}$ and in particular there is a single prime~$\p = (\zeta+\zeta^{-1})$ above~$2$.
The matrix~$\QQ$ of equation~(\ref{Qmat}) is no longer invertible locally at~$2$, so we have to work
in the original basis. Certainly the image of~$\Delta$ modulo~$\p$ has order~$2$, since~$T$ is trivial
modulo~$\p$. We already know that the image of~$\Delta$ in~$\SL_2(\OL/2)$ is isomorphic to~$D_{n/2}$.
If~$J \subset \OL$ is a proper ideal, let us define the following congruence subgroup:
$$\Gamma^0(J) =
\left(\begin{matrix} a & b \\ c & d \end{matrix} \right)
\equiv \left(\begin{matrix}  * & 0 \\  * &  * \end{matrix} \right) \bmod J.$$
Now if~$I$ and~$J$ are ideals, let us define
$$\Gamma^0(I,J) = \Gamma(I) \cap \Gamma^0(I J),$$
in words, matrices which are trivial modulo~$I$ but with an additional congruence
in the upper right hand corner by a further ideal~$J$.
The image of~$\Delta$ lands inside~$\Gamma^0(\p) = \Gamma^0(1,\p)$.  But we observe that the 
group~$\Gamma^0(2,\p)$ is normal in~$\Gamma^0(1,\p)$, and~$\Gamma^0(2^{r+1},\p)$ is normal
in~$\Gamma^0(2^{r},\p)$. Moreover, for~$r \ge 1$ there is an isomorphism of groups
$\Gamma^0(2^{r},\p)/\Gamma^0(2^{r+1},\p) \simeq M^0(\OL/2)$ given explicitly by considering
matrices of the form
\begin{equation}
\label{basiswithscalars}
I + 2^r \left( \begin{matrix} a  + b + c & b (\zeta + \zeta^{-1}) \\ c & a + b + c \end{matrix} \right)
\end{equation}
for~$a,b,c \in \OL/2$. 
Let~$\wQ$ inside this denote the subgroup~$(\OL/2)^2$ consisting of matrices of
the form
\begin{equation}
\label{moduleoldagain}
 \left(\begin{matrix} b+c & b(\zeta + \zeta^{-1})
 \\ c & b+c \end{matrix} \right).
\end{equation}
Then~$\wQ \simeq (\OL/2)^2$ as an abelian group,
but is is also preserved by the action of~$D_{n/2}$
by conjugation. Moreover, we also have an isomorphism
$$\Gamma^0(2^{r},\p)/\Gamma^0(2^{r+1},\p)  \simeq \OL/2 \oplus \wQ$$
of~$\F_2[D_{n/2}]$-modules.
The explicit action on this basis is given as follows: $R$, $T$, and~$U$ act
trivially on~$(a,0,0)$ and their action on~$(0,b,c)$  (by conjugation) satisfies:
\begin{equation}
\begin{aligned}
R.(0,b,c) - (0,b,c) = & \  b(0,0,\zeta+\zeta^{-1}) +  c(0,\zeta+\zeta^{-1}, (\zeta+\zeta^{-1})^2), \\
U.(0,b,c) - (0,b,c) = & \  b(0,0,\zeta+\zeta^{-1}), \\
T.(0,b,c) - (0,b,c) = & \  c(0,\zeta + \zeta^{-1},0)
 \end{aligned}
\end{equation}

\begin{lemma} \label{generate} $\wQ$ is generated as an~$\F_2[D_{n/2}]$-module by~$(0,1,0)$,
and~$(0,0,1)$.
\end{lemma}

\begin{proof}
Note that~$\OL/2$ is generated as an~$\F_2$ module by the powers of~$\pi = \zeta + \zeta^{-1}$.
But then the result is apparent by applying both~$U$ and~$T$ to get elements of the form~$(0,\pi^i,0)$
and~$(0,0,\pi^j)$ for any~$i$ and~$j$ which generate~$\wQ$.
\end{proof}

\begin{theorem} \label{2powertheorem} Assume~$n =2^k \ge 4$.
The group~$\Delta$ is a subgroup of~$\Gamma^0(\p) = \Gamma^0(1,\p)$.
Moreover:
\begin{enumerate}
\item The image of~$\Delta$ in~$\Gamma^0(1,\p)/\Gamma^0(2,\p)$ is isomorphic to~$D_{n/2}$.
\item 
The image of~$\Delta^0(2,\p)$ in~$\Gamma^0(2,\p)/\Gamma^0(4,\p)$ is isomorphic
to~$\Z/2\Z \oplus (\OL/2)^2$ given by~$\wQ$ in~(\ref{moduleoldagain})
together with~$\pm I$.
\item The image of~$\Delta^0(4,\p)$ in~$\Gamma^0(4,\p)$ is dense.
\end{enumerate}
In particular, if~$n=2^k$, so~$K$ has degree~$2^{k-2}$ and~$e=2^{k-2}$, the index is
$$\begin{aligned}
 \frac{[\Gamma:\Gamma^0(2,\p)]}{|D_{n/2}|} \cdot
\frac{|\OL/2|}{2} = &  \ \frac{ 2 |\SL_2(\OL/2)|}{n} \cdot \frac{|\OL/2|}{2} 
= \frac{  |\SL_2(\OL/2)| \cdot |\OL/2| }{n}  \\
= & \ \frac{6 \cdot 2^{3e-3} \cdot 2}{n} \cdot  2^{e-1} = 
\frac{3 \cdot 2^{4e-2} }{n} = 
3 \cdot 2^{2^k - k - 2}.
\end{aligned}
$$
\end{theorem}

\begin{proof}
Let us assume that~$n \ge 8$.
The image of~$\Delta$ in~$\SL_2(\OL/2)$ is isomorphic to~$D_{n/2}$.
As follows from the computation of  Section~\ref{2adicfirst},
the images of~$T^2$, $U^2$, and~$R^{d}$ with~$d=n/2$ all lie in~$\Gamma^0(2,\p)$,
and all have the form~(\ref{moduleoldagain}).
This proves that the image of~$\Delta$ in~$\Gamma^0(1,\p)/\Gamma^0(2,\p)$
is~$D_{n/2}$, and (since matrices of the form~(\ref{moduleoldagain})  together with~$-I$ 
are invariant
under conjugation) that the image of~$\Delta^0(2,\p)$ in~$\Gamma^0(2,\p)/\Gamma^0(4,\p)$
is satisfies the required containment. Hence it suffices to show that the image contains all
of this group. On the other hand, the images of~$-I$, ~$U^2$ and~$T^2$ in this group in the basis
of~$\OL/2 \oplus \wQ \simeq (\OL/2)^3$ given by~$(a,b,c)$ as in equation~\ref{basiswithscalars} have the form:
$$(1,0,0), (1,0,1), (1,1,0),$$
and these generate the desired space by Lemma~\ref{generate}.

We now turn to showing that the image of~$\Delta^0(4,\p)$ in~$\Gamma^0(4,\p)$
is dense. The argument will be quite similar to the proof of Lemma~\ref{satisfied}.
Consider~$I+2A$ and~$I+2B$ in~$\Gamma^0(2,\p)/\Gamma^0(4,\p)$. Then the commutator
$$[I+2A,I+2B]  = I + 4(AB-BA)$$
is well-defined as an element of~$\Gamma^0(4,\p)/\Gamma^0(8,\p)$, and is a scalar matrix.
However, this does not generate the entire subspace~$\OL/2$ of scalar matrices, since the condition
that~$A$ and~$B$ have upper right entry divisible by~$\pi$ implies that the image is the index
two subspace~$\pi \OL/2$. Thus we need to find some more diagonal matrices.
To this end, recalling that~$R^{n/2} = -I$ and~$4|n$, we note that
$$R^{n/4} =
\left(\begin{matrix} 
\displaystyle{
i  \cdot \frac{1+\zeta}{1-\zeta} 
}
&
\displaystyle{
2 i  \cdot \frac{1+\zeta}{1 - \zeta} .
} \\
2i \cdot \displaystyle{ \frac{ \zeta}{-1 + \zeta^2}}
&
-i \cdot \displaystyle{\frac{1 + \zeta}{1 - \zeta}}
 \end{matrix}
\right),$$
where~$i = \zeta^{2^{k-2}}$. Note that the entries here still lie in~$K$
though this is not transparent from how it is written. From this, we compute
further that
$$V = U^2  R^{n/4} U^2 R^{-n/4} = I +   4  \cdot \frac{(1+\zeta)^2}{(1 - \zeta)^2}
\left( \begin{matrix} -1 & -2 \\  \displaystyle{2 + \frac{2 \zeta}{(1+\zeta)^2}} & 5 \end{matrix} \right)$$
Certainly the valuation of~$1+\zeta$ and~$1-\zeta$ are equal.
Moreover, the valuation of~$2$ is greater than the valuation of~$(1+\zeta)^2$
since~$n = 2^k \ge 8$. Hence the image of~$V$ in~$\Gamma^0(4,\p)/\Gamma^0(4 \p,\p)$
is a non-trivial scalar.

Now we are ready to prove that~$\Delta^0(4,\p)/\Delta^0(8,\p)$ is isomorphic
 to~$\Gamma^0(4,\p)/\Gamma^0(8,\p)$. Take~$A$ to be any element in~$\wQ$,
 so that we know~$I+2A \in \Delta^0(2,\p)/\Delta^0(4,\p)$. Then we find that
 \begin{equation}
 \label{squareagain}
 (I+2A)^2 = I + 4A + 4A^2 = 1 + 4B, \quad B \in \OL/2 \oplus \wQ
 \end{equation}
 is  a well-defined element of~$\Delta^0(4,\p)/\Delta^0(8,\p)$. Moreover,
 we find that~$A^2$ is scalar. If we take~$A$
 to be divisible by~$\pi$, then certainly~$A^2$ is divisible by~$\pi$, and
 thus since we have seen above using commutators that everything in~$\pi \OL/2$
 lies in the image of~$\Delta^0(4,\p)$ we see that everything in~$\pi \OL/2 \oplus \pi \wQ$
 lies in the image. In particular, we are reduced to showing that~$\Delta^0(4,\p)$
 surjects onto~$\Gamma^0(4,\p)/\Gamma^0(4\p,\p)$. The image of~$V$ gives 
 a non-trivial scalar matrix. But now by~(\ref{squareagain}) we see that anything in~$\wQ \bmod \pi$
 lies in the image up to this scalar, and hence everything is in the image.
 Thus we are done.
 Now the required identification
 $$\Delta^0(2^r,\p)/\Delta(2^{r+1},\p) \simeq
  \Gamma^0(2^r,\p)/\Gamma^0(2^{r+1},\p)$$
   follows from the case~$r=2$ by induction
  and the fact that for~$I+2^r A \in \Delta^0(2^r,\p)$ we have
  (now using that~$r \ge 2$) that~$(I+2^r A)^2 = I + 2^{r+1}A
  \in \Delta^0(2^{r+1},\p)$.
  Finally it remains to consider the case~$n=4$. In this case, the group~$\Delta_4$  is
  just the group~$\Gamma^0(2)$ of
 finite index~$3$ inside~$\SL_2(\Z)$, and the results hold (alternatively one can use the same
  proof as above after computing any suitable~$V$). \end{proof}

\section{The invariants $\delta$  and~$\kappa$}
\label{deltakappa}

Let us now consider the invariants~$\delta$ and~$\kappa$ introduced in~\cite{CM}.
As previously mentioned, for any cusp~$x \in \PP^1(K)$, the invariant~$\kappa(x)$
is defined to be~$0$ if~$x$ is in the orbit of~$\infty$ under~$\Delta$ and~$1$ otherwise.
(In particular, if~$n$ is odd and there is only one class of cusp, then~$\kappa(x)$ always
equals zero.) To prove that~$\kappa$ is not a congruence invariant for some even~$n$,
it suffices to construct cusps~$x$ adelically close to~$\infty$ such that~$\kappa(\infty) = 1$.
We phrase this more explicitly as follows. Suppose that
$$\gamma = \left( \begin{matrix} a & b \\ c & d \end{matrix} \right) \in \Delta.$$
Then~$\gamma(0) = b/d$ is  a cusp equivalent to~$0$, so  to
show that~$\kappa$ is non-congruence it suffices to find,
for any integer~$N$, an element with~$d \equiv 0 \bmod N$,
so~$[b/d] = [\infty]$ in~$\PP^1(\OL/N)$ but~$\kappa(b/d) = 1$.

The second invariant~$\delta$ is defined as follows.
Let~$n' =n$ if~$n$ is odd and~$n/2$ otherwise.
(The integer~$n'$
was denoted by~$m$ in~\cite{CM} but that conflicts with our
notation.)
The 
 image
of~$\Delta$ in~$\SL_2(\OL/2)$ is identified 
with~$D_{n'} = \langle r,t \rangle
\subset \mathrm{Sym}(\Z/m)$,
where~$r(x) = x+1$
and~$t(x) = -x$.
 The cusps of~$\Delta$ 
equivalent to $\infty$
 are given by
the coset~$\Delta/\langle T \rangle$, and the map~$\Delta \rightarrow \SL_2(\OL/2)$
sending~$T$ to~$t \in D_{n'}$ 
maps~$\Delta/\langle T \rangle$ to~$D_{n'}/\langle t \rangle$,
and then the latter is 
identified
with~$\Z/n' \Z$ by sending~$g$ to~$g(0)$;
this is the definition of~$\delta$ on the cusps
equivalent to~$\infty$.
Note that~$\delta$ 
can also be considered as
 a map
on~$\Delta$ itself by the formula
\begin{equation}
\label{deltaformula}
\delta(\gamma) = \delta(\gamma(\infty)).
\end{equation}

\subsection{Some Reductions} \label{curtadded}
In~\cite{CM},
the map~$\delta$ can be extended to
a function defined on all the cusps.
The proof of Theorem~\ref{curty} in this case 
 can be
reduced to the
rest of Theorem~\ref{curty}
as follows.
If $\delta$ is not congruence as a function on the
cusps equivalent to $\infty$, it is certainly not congruence as an invariant
on all of the cusps. Thus, since~$\delta$ \emph{is} a congruence
invariant on all cusps when $n=2^k$ (by \cite[Cor~1.5]{CM}),
and since there is only one orbit of cusps when $n$ is odd,
it suffices to consider the case when $2$ exactly divides~$n$.
If~$n=4g+2$, 
 it follows
from Theorem~\cite[3.2]{CM} that~$\delta$
(as a function on all the cusps) is a congruence invariant if and only if~$\kappa$
is congruence, and then the result follows from our claims
concerning~$\kappa$.

From this point onwards, we now regard~$\delta$ as  function
on the cusps equivalent to~$\infty$, or equivalently as a function
on the group~$\Delta$ as described in equation~\eqref{deltaformula}.

\subsection{The strategy}
As noted in~\cite{CM}, the map~$\delta \mod 2$ is a homomorphism but~$\delta$
is not a homomorphism in general. The map~$\delta$ on~$\Delta$
clearly only depends on the congruence class of~$\gamma \in \Delta$.
However, as a map on cusps, this is no longer necessarily true.
To show (for a particular even integer~$n$)  that~$\delta$ is not a congruence invariant
of the cusps equivalent to~$\infty$, it suffices
to find, for any integer~$N$, a~$\gamma \in \Delta$ such that:
\begin{enumerate}
\item $\gamma \equiv R^i \bmod 2$ where~$i \not\equiv 0 \bmod n' = n/2$.
\item The cusp~$\gamma(\infty) = [a:c]$ is equal to~$[1:0] \in \PP^1(\OL/N)$,
equivalently, that~$c \equiv 0 \bmod N$.
\end{enumerate}
The point is that the first condition implies that~$\delta(\gamma(\infty)) = i \bmod n/2$.
Since (by~\cite[Cor~1.7]{CM}) the invariant~$2 \delta$ \emph{is} congruence, this should
only be possible if~$4|n$ and~$i = n/4 \bmod n/2$.
Note that by the proof of~\cite[Theorem~1.13]{CM},  it indeed suffices
to find such a~$\gamma$ for~$N$ any power of~$2$, although we don't need to use this fact.

 \begin{theorem} \label{maindeltatwo}
The map~$\delta \mod 2^{k-1}$ is a not a congruence invariant 
on the cusps equivalent to~$\infty$
if~$n = 2^k m$ where~$k \ge 2$ and~$m > 1$ is odd.
\end{theorem} 
 
 \begin{proof}
 We begin by
considering the element
$$\gamma_1:=R^{2^{k-2}m}  \equiv I + \left( \begin{matrix} i-1 & 0 \\ 0 & i-1 \end{matrix} \right) \mod 2.$$
Note that~$n/4 = 2^{k-2}m$.
By construction, we therefore have
$$\delta(R^{2^{k-2}m}) \equiv n/4 \bmod n/2.$$
Since~$\delta$ on~$\Delta$ depends only on the image in~$\Delta/\Delta(2) \subset \SL_2(\OL/2)$,
it follows that any lift of this element has this property.
We now consider~$\gamma_1 = R^{n/4} \mod 4$. 
By Theorem~\ref{2part} (we are assuming~$n$ is not a power of~$2$),
 we can inductively find $\eta_r$ in~$\Delta(2^r)/\Delta(2^{r+1})$
so that~$\gamma_{r+1}:=\eta_r \gamma_r \mod 2^{r+1}$ also has zero in the lower left corner.
Inductively, we construct a limit point~$\gamma_{\infty} \in \wDelta_2$ with a zero in the lower left hand corner.
Similarly,  we deduce
from Theorem~\ref{awayfrom2}  that we may find~$\gamma_{\infty} \in \wDelta$ with the same property.
(The condition that~$4|n$ implies that the closure of the image away from~$2$  is everything.)
Now take any~$\gamma \in \Delta$ which is congruent to~$\gamma_{\infty}$ modulo~$N$ for any integer~$N$,
we deduce that~$\gamma(\infty) = [a:c]$ reduces to~$\infty$ in~$\PP^1(\OL/N)$, 
but~$\delta(a/c) = \delta(\gamma) = n/4 \bmod n/2$, and hence~$\delta$ is not a congruence invariant modulo~$N$
for any integer~$N$.
\end{proof}

We now turn to~$\kappa$. We begin with a preliminary lemma.

\begin{lemma} \label{kappa2} 
Assume that~$n=2m$ where~$m$ is odd.
There exists an element~$\gamma_{\infty}$ in the closure
of~$\Delta$ in~$\SL_2(\OL \otimes \Z_2)$
of the form
$$\left( \begin{matrix} 0 & * \\ * & * \end{matrix} \right) \bmod 2^k.$$
\end{lemma}

\begin{proof}
We proceed by induction. Since~$n \equiv 2 \mod 4$, we may
take~$i = (n+2)/4$ and then
$$\gamma_1 = R^i =  \frac{1}{\zeta^i(1-\zeta^2)}
\left( \begin{matrix} (1 + \zeta)(\zeta^{2i} - \zeta) &
(1 + \zeta)^2(1-\zeta^i)^2 \\
\zeta(\zeta^i - 1)^2 & \zeta(1 + \zeta)(\zeta^{2i} - \zeta^{-1})
\end{matrix} \right).$$
Since~$m > 1$,  the leading factor is a unit.
But~$\zeta^{2i-1} = \zeta^{n/2} = -1$, so~$\zeta^{2i} - \zeta
\equiv 0 \bmod 2$. Hence~$\gamma_1$ is such an element modulo~$2$.
We now consider~$\gamma_1 \bmod 4$. 
By Theorem~\ref{2part} (we are assuming~$n$ is not a power of~$2$),
we can adjust~$\gamma_1 \bmod 4$ by any element~$\eta_1$ in~$\wQ$,
which have the form:
$$
\left(\begin{matrix} b+c & b(\zeta + \zeta^{-1})
 \\ c & b+c \end{matrix} \right) 
$$
In particular, since we can choose~$b+c \in \OL/2$ to be anything.
 we can ensure that~$\gamma_2 = \eta_1 \gamma_1$ has a~$0$
in the first entry. 
Just as in the proof of Theorem~\ref{maindeltatwo},
we proceed by induction to find the required element~$\gamma_{\infty}$.
\end{proof}

\begin{theorem} \label{kappacong} The function~$\kappa$ is not a congruence invariant if~$n = 2m$ where~$m$
is odd and divisible by at least two prime factors.
\end{theorem}

\begin{proof} 
We claim that for any integer~$N$, there exists a~$\gamma \in \Delta$ 
such that
$$\gamma  = \left(\begin{matrix} a & b \\ c & d \end{matrix} \right)
\equiv \left( \begin{matrix} 0 & * \\  * & * \end{matrix} \right) \bmod N.$$
By Theorem~\ref{awayfrom2}, it suffices to show that this can be achieved when~$N$ is any prime power.
By Theorem~\ref{awayfrom2} and the assumption that~$n$ has at least three prime factors,
it follows that~$\wDelta_p = \wGamma_p$, and hence contains the image of~$\Delta$ contains the element
$$\left( \begin{matrix} 0 & 1 \\ -1 & 0 \end{matrix} \right)$$
modulo any power of an odd prime~$p$. On the other hand, by Lemma~\ref{kappa2},
there exists an element of the required form modulo any power of~$2$, which proves the claim.
But that means that the cusp~$\gamma(\infty) = (a,c) \in \PP^1(\OL/N)$
maps to the image of~$0$ in~$\PP^1(\OL/N)$, and so~$\kappa$ is not a congruence
invariant modulo~$N$. Since this holds for any integer~$N$, we are done.
\end{proof}

\bibliographystyle{amsalpha}
\bibliography{Chaos}

\end{document}